\documentclass[12pt]{amsart}

\usepackage{anysize}
\marginsize{3cm}{3cm}{3cm}{3cm}

 \usepackage{bbold}
 \usepackage{amsmath}
 \usepackage{tikz}
\usetikzlibrary{patterns}
\usepackage{pgfplots,caption}
\pgfplotsset{compat=1.12}
\usepgfplotslibrary{fillbetween}
\usepackage{cool}
\usepackage{color}
\renewcommand{\Re}{\operatorname{Re}}

\usepackage{amsthm}
\newtheorem{theorem}{Theorem}[section]

\newtheorem{lemma}[theorem]{Lemma}
\newtheorem{proposition}[theorem]{Proposition}
\theoremstyle{definition}

 \theoremstyle{remark}

\usepackage{enumerate}
\usepackage[utf8]{inputenc}

\title{A two variable zeta function associated to the space of binary forms of degree $d$}
\author{Eun Hye Lee}
\address{Department of Mathematics, Texas Christian University, TCU Box 298900, Fort Worth, TX 76129}
\email{eun.hye.lee@tcu.edu}

\author{Ramin Takloo-Bighash}
\address{Dept. of Math, Stat, and Comp. Sci, University of Illinois at Chicago, 851 S. Morgan St, Chicago, IL 60607}
\email{rtakloo@uic.edu}
\date{September 2023}

\begin{document}

\maketitle

\begin{abstract}
    In this paper we prove the analytic continuation of a two variable zeta function defined using the vector space of binary forms of degree $d$ to the entire two dimensional complex space as a meromorphic function. 
\end{abstract}

\section{Introduction}

Let $X_d$ be the vector space of all binary $d$ forms of degree $d$, i.e., the vector space of all polynomials of the form 
\begin{equation}\label{eq:general}
F(X, Y) = \sum_{r=0}^d a_r X^r Y^{d-r}
\end{equation}
with real coefficients. Let $X_d^+$ be the collection of forms with integral coefficients such that $a_d  >0$. Also, let $\Gamma_\infty=\left<\begin{pmatrix} 1&1\\0&1\end{pmatrix}\right>$ be the upper triangular unipotent elements of $\operatorname{SL}_2(\mathbb{Z})$. Then we have a left action of $\Gamma_\infty$ on $X_d^+$, given by $\gamma\circ F(X,Y)=F((\gamma\circ (X,X)^T)^T)$.
We note that under the action of $\Gamma_\infty$ every form is equivalent to form $F$ in Equation \eqref{eq:general} such that $0 \leq a_{d-1} \leq da_d -1$. 

\begin{lemma}
		Let $\displaystyle F(X,Y)=\sum_{r=0}^d a_r X^rY^{d-r}$. Then, $\displaystyle I_2= (d-1) a_{d-1}^2-2da_da_{d-2}$ is invariant under $\displaystyle \Gamma_\infty=\left<\begin{pmatrix}1&1\\0&1\end{pmatrix}\right>$.
	\end{lemma}
	\begin{proof}
		Since $\displaystyle F(X+tY,Y)=\sum_{r=0}^d\  a_r\ (X+tY)^r\ Y^{d-r}$, we have 
		\begin{align*}
			\frac{\partial F}{\partial t}(X+tY,Y)=&\sum_{r=0}^d\ a_r\ r\  (X+tY)^{r-1}\ Y^{d-r+1}\\
			\frac12\frac{\partial^2F}{\partial t^2}(X+tY,Y)=&\sum_{r=0}^d\ a_r\ \frac{r(r-1)}2\ (X+tY)^{r-2}\ Y^{d-r+2}.
		\end{align*}
		Hence, 
		\begin{align*}
			&F(X+tY,Y)|_{t=0}+s\ \frac{\partial F}{\partial t}(X+tY,Y)|_{t=0}+\frac{s^2}2\ \frac{\partial^2 F}{\partial t^2}(X+tY,Y)|_{t=0}\\&=\sum_{r=0}^d\ a_r\ X^r\ Y^{d-r}+s\ \sum_{r=0}^d\ a_r\ r\ X^{r-1}\ Y^{d-r+1}+s^2\ \sum_{r=0}^d\ a_r\ \frac{r(r-1)}2\  X^{r-2}\ Y^{d-r+2}\\&=a_d\ X^n+(a_{d-1}+s\ a_d\ d)X^{d-1}\ Y\\&\phantom{=}+\sum_{r=0}^{d-2}\left(a_{d-r}+s\ a_{d-r+1}(d-r+1)+s^2\ a_{d-r+2}\ \frac{(d-r+2)(d-r+1)}2\right)X^{d-r}\ Y^r.
		\end{align*}
		Therefore, 
		$$
			(a_{d-1}+s\ a_d\ d)^2-\frac{2d}{d-1}\ a_d \left(a_{d-1}+s\ a_{d-1}(d-1)
   +s^2\ a_d\ \frac{d(d-1)}2\right)=a_{d-1}^2-\frac{2d}{d-1}\ a_d\ a_{d-2}.
		$$
	\end{proof}

In analogy with the zeta function considered in \cite{LTB} we set, at first formally, for $(s_1, s_2) \in \mathbb C^2$, 
$$
Z_d(s_1, s_2) = \sum_{\substack{a, b , c \in \mathbb Z \\ a >0, 0 \leq b < da \\ 2d a c - (d-1)b^2 >0 \\ 
2d a c - (d-1)b^2 \text{ square-free, odd}}} \frac{1}{a^{s_1}(2d a c - (d-1)b^2)^{s_2}}.
$$

\begin{theorem}\label{thm:main}
    The function $Z_d(s_1, s_2)$ converges absolutely in $\Re s_1, \Re s_2 \gg 0$, and has an analytic continuation to a meromorphic function on all of $\mathbb C^2$.
\end{theorem}

In fact we will prove a more general result, see Theorem \ref{thm:main} below. The proof we present in the following section is an adaptation of the proof of the main theorem of \cite{LTB}. 

\

 The zeta function in Theorem \ref{thm:main} lacks the typical symmetries that zeta functions with analytic continuation enjoy. In particular we do not know if it satisfies a functional equation. In general, whenever a zeta function has an analytic continuation it is for a very good reason. At present we do not have a conceptual explanation for why the zeta function considered here continues to all of $\mathbb C^2$. It would be desirable to find such an explanation. 

 The second author is partially supported by a grant from the Simons Foundation. We thank Evan O'Dorney for suggesting that we extend our results from \cite{LTB} to general $d$, and Takashi Taniguchi for catching an error in an earlier version of this paper.  

The paper is organized as follows. After this introduction, we define a general zeta function in \S \ref{sect:zeta} and prove its analytic properties. The main theorem of the paper is included at the end \S \ref{sect:zeta} as Theorem \ref{thm:main}. 

\section{A two variable zeta function}\label{sect:zeta}

Let $A, B$ be a pair of coprime integers. For $(s_1, s_2) \in \mathbb C^2$, set, at first formally, 
\begin{equation}\label{zetaAB}
Z_{A, B}(s_1, s_2) = \sum_{\substack{ a, b, c \in \mathbb Z 
\\ a>0, 0 \leq b < Aa \\  A a c - B b^2 > 0, \text{ odd, square-free}}} \frac{1}{a^{s_1} (A ac - B b^2)^{s_2}}. 
\end{equation}

It is easy to see that the above sum is formally equal to $$
Z_{A, B}(s_1, s_2) = \sum_{m, n =1 \atop n \text{ odd, square-free}}^\infty \frac{C_{A, B}(m, n)}{m^{s_1}n^{s_2}}
$$
with 
$$
C_{A, B}(m, n) = \# \{ x \mod m A \mid B x^2 \equiv - n \mod m A\}.
$$
Since $C_{A, B}(m, n)$ is at most $mA$, the absolute convergence of $Z(s_1, s_2)$ for $\Re s_1, \Re s_2$ large is immediate. Until otherwise noted, we will proceed within the domain of absolute convergence. 

\

Next, we write 
$$
Z_{A, B}(s_1, s_2) = \sum_{d \mid B} \sum_{\substack{m, n =1 \\ \gcd(B, m)=\delta \\ 
n \text{ odd, square-free}}}^\infty\frac{C_{A, B}(m, n)}{m^{s_1}n^{s_2}}.
$$
In the inner sum, unless $\delta \mid n$, $C_{A, B}(m, n) =0$. Next, when $\delta \mid m, \delta \mid n$ we have 
\begin{align*}
C_{A, B}(m, n) & = \# \{ x \mod m A \mid B x^2 \equiv - n \mod m A\} \\ 
&= \delta \# \{ x \mod \frac{m}{\delta} A \mid \frac{B}{\delta} x^2 \equiv - \frac{n}{\delta} \mod \frac{m}{\delta} A\} \\ 
& = \delta C_{A, \frac{B}{\delta}}(\frac{m}{\delta}, \frac{n}{\delta}). 
\end{align*}
After a change of variables we write 
$$
Z_{A, B}(s_1, s_2)= \sum_{\delta \mid B} \frac{1}{\delta^{s_1 + s_2 -1}} \sum_{\substack{m, n =1  \\ \gcd(m, \frac{B}{\delta})=1 \\ n \text{ odd, square-free}}}^\infty \frac{C_{A, \frac{B}{\delta}} (m, n)}{m^{s_1} n^{s_2}}.
$$

Since we are interested in analytic continuation, it suffices to prove the analytic continuation of each of the inner summands. For that reason replacing $B/\delta$ with $B$, for a pair of coprime integers $A, B$ we set 
$$
\zeta_{A, B}(s_1, s_2) = \sum_{\substack{m, n =1  \\ \gcd(m, B)=1 \\ n \text{ odd, square-free}}}^\infty \frac{C_{A, B} (m, n)}{m^{s_1} n^{s_2}}
$$

Since we need to keep track of moduli, given a modulus $k$, whenever $\gcd(x, k)=1$, we denote the multiplicative inverse of $x$ modulo $k$ by $f_k(x)$.   We note that $f_{k}(x)$ can always be chosen to be represented by an odd number and we will do this. Indeed, if $2\mid k$, then $x$ is odd as have assumed $\gcd(x, k)=1$, and that means that since $f_{k}(x)$ will be odd for any choice of the representative. On the other hand, if $k$ is odd, and $f_{k}(x)$ is represented by an even number, then we will replace $f_k(x)$ by $f_k(x) + k$ to obtain an odd number. 

With this convention, since $\gcd(B, mA)=1$, we have 
\begin{align*}
C_{A, B} (m, n) & = \# \{ x \mod m A \mid x^2 \equiv - f_{mA}(B) n \mod m A\} \\ 
&= C(mA, - f_{mA}(B) n), 
\end{align*}
where as in \S 2.1 of \cite{LTB} for integers $m, n $
$$
C(m, n) = \# \{ x \mod m  \mid x^2 \equiv n \mod m\}. 
$$
Recall the following proposition from \cite{LTB}:

\begin{proposition} The following properties hold. \begin{enumerate}[i)]\item For any fixed n, $C(m,n)$ is a multiplicative function in $m$. In particular, $C(1,n)=1$ for all $n$.
\item If any prime $p\ne 2$ and $p\nmid n$, then $C(p^{\alpha},n)=1+\left(\dfrac{n}{p}\right)$ for $\alpha>0$.
\item If $p=2$ and $n$ is odd, then for $\alpha>0$, $$C(2^{\alpha},n)=\begin{cases} 1 & \alpha=1,\\ 2 & \alpha=2, n\equiv 1 \mod 4, \\ 4 & \alpha\ge 3, n\equiv 1\mod 8, \\0 & \mbox{otherwise.}\end{cases}$$
\item If $n=p^r n_0$ with $p\nmid n_0$, then for $\alpha>0$, $$C(p^{\alpha},p^rn_0)=\begin{cases} p^{\left\lfloor\frac{\alpha}{2}\right\rfloor} & r\ge\alpha,  \\ p^{\frac{r}2}C(p^{\alpha-r},n_0) & r<\alpha, r \mbox{ even},\\ 0 & \mbox{otherwise.}\end{cases}$$ \end{enumerate}
\end{proposition}

\begin{proof} See Proposition 2.2 of \cite{LTB}. 
\end{proof} 

We apply the proposition to compute the value of $C(mA, - f_{mA}(B)n)$. We write $mA= p_1^{\alpha_1} \cdots p_r^{\alpha_r}$.  By multiplicativity, we have 
$$
C(mA, - f_{mA}(B)n) = \prod_i C(p_i^{\alpha_i}, - f_{mA}(B)n).
$$
If $p_i \nmid n$ and $p \ne 2$, then 
$$
C(p_i^{\alpha_i}, - f_{mA}(B)n) = 1 + \left(\frac{- f_{mA}(B)n)}{p_i} \right) = 1 +  \left(\frac{f_{mA}(B)}{p_i} \right) \left(\frac{- n}{p_i} \right). 
$$
Since $B f_{mA}(B) \equiv 1 \mod mA$ and $p_i \mid mA$, this means $B f_{mA}(B) \equiv 1 \mod p_i$. Consequently, $(f_{mA}(B)/p_i) = (B/p_i)$. We have 
$$
C(p_i^{\alpha_i}, - f_{mA}(B)n) = 1 + \left(\frac{-Bn}{p_i} \right).
$$
If $p_i \| n$ and $p_i \ne 2$, then 
$$
C(p_i^{\alpha_i}, - f_{mA}(B)n) =  \begin{cases} 
1 & \alpha_i=1 \\
0 & \alpha_i > 1.
\end{cases}
$$
Since $n$ is assumed to be square free we do not need to consider the cases where $p_i^2 \mid n$.

If $p_i=2$, let $2^\alpha \| mA$. Then since $2^\alpha \mid mA$, $B f_{mA}(B) \equiv 1 \mod 2^\alpha$. This implies 
$$C(2^{\alpha},-n f_{mA}(B))=\begin{cases} 1 & \alpha=1,\\ 2 & \alpha=2, n\equiv -B \mod 4, \\ 4 & \alpha\ge 3, n\equiv -B\mod 8, \\0 & \mbox{otherwise.}\end{cases}$$

We note the following important consequence of the above computations: 

$$
C(mA, - f_{mA}(B)n) = \prod_i C(p_i^{\alpha_i}, - f_{p^i}(B)n).
$$

Next, we write 

\begin{align*}
\zeta_{A, B}(s_1, s_2) & = \sum_{\substack{m, n =1  \\ \gcd(m, B)=1 \\ n \text{ odd, square-free}}}^\infty \frac{C_{A, B} (m, n)}{m^{s_1} n^{s_2}} \\
&= \sum_{\substack{n=1 \\ n \text{ odd, square-free}}}^\infty
\frac{1}{n^{s_2}} \sum_{\substack{m=1  \\ \gcd(m, B)=1 }}^\infty \frac{C (mA, -n f_{mA}(B))}{m^{s_1}} \\
& = \sum_{\substack{n=1 \\ n \text{ odd, square-free}}}^\infty
\frac{Z_{n, A, B}(s_1)}{n^{s_2}}
\end{align*}
where for any $s \in \mathbb C$ we have set 
$$
Z_{n, A, B}(s) = \sum_{\substack{m=1  \\ \gcd(m, B)=1 }}^\infty \frac{C (mA, -n f_{mA}(B))}{m^{s}}. 
$$
As in Proposition 2.3 of \cite{LTB} the zeta function $Z_{n, A, B}$ has an Euler product expansion of the form 
$$
Z_{n, A, B} (s) = \prod_{p \nmid B} Z_{n, A, B, p}(s)
$$
with 
$$
Z_{n, A, B, p}(s) = \sum_{k =0}^\infty   \frac{C (p^{k+ \text{ord}_p(A)}, -n f_{p^{k+ \text{ord}_p(A)}}(B))}{p^{ks}}.
$$

We now compute the local zeta functions for various primes $p$.  We will be adapting the computations of \S 2.1.1 of \cite{LTB}. 

\

If $p \nmid 2 A$, then 
$$
Z_{n, A, B, p}(s)= \dfrac{1-p^{-2s}}{1-p^{-s}}\cdot\dfrac1{1-p^{-s}\left(\frac{-nB}p\right)}.
$$

\

If $p \mid A$, $p \nmid n$, $p$ odd, then 
$$
Z_{n, A, B, p}(s)= \left(1+\left(\dfrac{-nB}p\right)\right)\dfrac1{1-p^{-s}}.
$$

\

If $p\mid A$, $p\| n$, $p$ odd, then if we write $n f_{p^{k+ \text{ord}_p(A)}}(B) =p n_0$
\begin{align}
    Z_{n, A, B, p}(s) & = \sum_{k =0}^\infty   \frac{C (p^{k+ \text{ord}_p(A)}, -pn_0)}{p^{ks}} \\
    &= \begin{cases}
1 & p \| A \\ 
0 & p^2 | A.
    \end{cases} \label{zero}
\end{align}

\

Finally, let $p=2$. We note that this means $B$ is odd and for each $k \geq 1$, $f_{2^{k+ \text{ord}_2(A)}}(B)$ is also odd. We recognize a few cases based on $\text{ord}_2(A)$. 

\

If $\text{ord}_2(A) =0$, then we have the following possibilities: 

\

\begin{itemize}

\item If $n \equiv -B \mod 8$, then 
$$
Z_{n, A, B, 2}(s) =\dfrac{1-2^{-2s}}{1-2^{-s}}\cdot\dfrac{2\cdot2^{-2s}-2^{-s}+1}{1-2^{-s}}. 
$$

\item If $n \equiv -B + 4 \mod 8$, then 
$$
Z_{n, A, B, 2}(s) =1+\dfrac1{2^s}+\dfrac2{2^{2s}}.
$$

\item If $n \equiv -B + 2$ or $n \equiv -B + 6 \mod 8$, then 
$$
Z_{n, A, B, 2}(s) =1+\dfrac1{2^s}.
$$
\end{itemize}

If $\text{ord}_2(A) =1$, then we have the following cases: 

\

\begin{itemize}
    \item If $n\equiv -B \mod 8$, then 
    $$
    Z_{n, A, B, 2}(s) = 1 + \dfrac{1}{2^s} + \dfrac{4 \cdot 2^{-2s}}{1-2^{-s}}.
    $$
    \item If $n \equiv -B + 4 \mod 8$, then 
    $$
    Z_{n, A, B, 2}(s) = 1 + \dfrac{2}{2^s}.
    $$
    \item If $n \equiv -B + 2$ or $n \equiv -B + 6 \mod 8$, then 
    $$
    Z_{n, A, B, 2}(s) = 1. 
    $$
\end{itemize}

\ 

If $\text{ord}_2(A) =2$, then we have the following cases:

\

\begin{itemize}
    \item If $n\equiv -B \mod 8$, then 
    $$
    Z_{n, A, B, 2}(s) = 2 + \dfrac{4 \cdot 2^{-s}}{1-2^{-s}}.
    $$
    \item If $n \equiv -B + 4 \mod 8$, then 
    $$
    Z_{n, A, B, 2}(s) = 2.
    $$
    \item If $n \equiv -B + 2$ or $n \equiv -B + 6 \mod 8$, then 
    $$
    Z_{n, A, B, 2}(s) = 0. 
    $$
\end{itemize}

\

If $\text{ord}_2(A) \geq 3$, then we have the following cases:

\

\begin{itemize}
    \item If $n\equiv -B \mod 8$, then 
    $$
    Z_{n, A, B, 2}(s) = \dfrac{4}{1-2^{-s}}.
    $$
    \item If $n \equiv -B + 2$, $-B + 4$, $-B + 6 \mod 8$, then 
    $$
    Z_{n, A, B, 2}(s) = 0.
    $$
\end{itemize}

The upshot of the above computation is the following proposition:
\begin{proposition}
    For any $n$ and $p \nmid B$ there is a rational function $\gamma_{n, A, B, p}(X)$ such that 
$$
Z_{n, A, B, p}(s)= \gamma_{n, A, B, p}(p^{-s}) \cdot \dfrac{1-p^{-2s}}{1-p^{-s}}\cdot\dfrac1{1-p^{-s}\left(\frac{-nB}p\right)}, 
$$
and if $p \nmid 2A$, then $\gamma_{n, A, B, p}(X)=1$. Furthermore, if $n_1 \equiv n_2 \mod 8A$, then for all $p$, 
$$
\gamma_{n_1, A, B, p}(X)=\gamma_{n_2, A, B, p}(X). 
$$
\end{proposition}
The above proposition should be compared with Proposition 2.4 of \cite{LTB}.  It should also be noted that if $B$ is odd, then the period $8A$ can be replaced by $A$. 

\

Going back to the two variable zeta function we get 
$$
\zeta_{A, B}(s_1, s_2) = \frac{\zeta(s_1)}{\zeta(2s_1)}\sum_{n \text{ odd, square-free}} a_{n, A, B}(s_1) \frac{L_{2AB}\left(\left(\frac{-nB}{\cdot} \right), s_1\right)}{n^{s_2}}, 
$$
where 
$$
a_{n, A, B}(s) = \prod_{p \mid 2A} \gamma_{n, A, B, p}(p^{-s}), 
$$
and $L_{2AB}(\left(\frac{-nB}{\cdot} \right), s_1)$ is the Dirichlet $L$-function of the character $\psi(k)= \left(\frac{-nB}{k} \right)$ with Euler factors corresponding to the primes $p$ dividing $2A$ removed.

In order to prove the analytic continuation of the above zeta function we set  
$$
\tilde \zeta_{A, B}(s_1, s_2)= \frac{\zeta(s_1)}{\zeta(2s_1)}\sum_{\substack{n \text{ square-free} \\ \gcd(n, 2A)=1}} a_{n, A, B}(s_1) \frac{L_{2AB}\left(\left(\frac{-nB}{\cdot} \right), s_1\right)}{n^{s_2}}. 
$$
Set 
$$
\delta_j(n) = \begin{cases} 1 & n \equiv j \mod 8 A \\ 
0 & n \not\equiv j \mod 8 A. 
\end{cases}
$$
By the orthogonality of characters, if $\gcd(j, 8A)=1$, 
$$
\delta_j(n) = \frac{1}{\phi(8A)} \sum_{\chi} \chi(j)^{-1} \chi(n), 
$$
where the sum is over all Dirichlet characters modulo $8A$.  We then have 
\begin{align*}
    \tilde \zeta_{A, B}(s_1, s_2) & = \frac{\zeta(s_1)}{\zeta(2s_1)} \sum_{\substack{n \text{ square-free} \\ \gcd(n, 2A)=1}} a_{n, A, B}(s_1) \frac{L_{2AB}\left(\left(\frac{-nB}{\cdot} \right), s_1\right)}{n^{s_2}} \\ 
    & = \frac{\zeta(s_1)}{\zeta(2s_1)} \sum_{\substack{n \text{ square-free}}} \sum_{j \in (\mathbb Z / 8 A \mathbb Z)^\times} \delta_j(n)a_{j, A, B}(s_1) \frac{L_{2AB}\left(\left(\frac{-nB}{\cdot} \right), s_1\right)}{n^{s_2}}\\
    & = \frac{\zeta(s_1)}{\zeta(2s_1)} \sum_{\substack{n \text{ square-free}}} \sum_{j \in (\mathbb Z / 8 A \mathbb Z)^\times} \frac{1}{\phi(8A)} \sum_{\chi} \chi(j)^{-1} \chi(n)a_{j, A, B}(s_1) \frac{L_{2AB}\left(\left(\frac{-nB}{\cdot} \right), s_1\right)}{n^{s_2}}\\
    & = \frac{\zeta(s_1)}{\zeta(2s_1)} \frac{1}{\phi(8A)} \sum_{\chi}\sum_{j \in (\mathbb Z / 8 A \mathbb Z)^\times} \chi(j)^{-1}a_{j, A, B}(s_1)\sum_{\substack{n \text{ square-free}}}\chi(n)\frac{L_{2AB}\left(\left(\frac{-nB}{\cdot} \right), s_1\right)}{n^{s_2}}. 
\end{align*}
The analytic continuation of the innermost zeta function 
$$
\sum_{\substack{n \text{ square-free}}}\chi(n)\frac{L_{2AB}\left(\left(\frac{-nB}{\cdot} \right), s_1\right)}{n^{s_2}}
$$
to a meromorphic function on all of $\mathbb C^2$ is the main result of \S 2.2-2.3 of \cite{LTB}. Consequently we get the following proposition: 
\begin{proposition}\label{coprime}
    The zeta function $\tilde \zeta_{A, B}(s_1, s_2)$ has an analytic continuation to a meromorphic function on all of $\mathbb C^2$. 
\end{proposition}

Next we treat the zeta function $\zeta_{A, B}(s_1, s_2)$. Write $A=A_1A_2$ with $A_1$ square-free, $A_2$ squarefull, i.e., for all $p$, $p\mid A_2$ implies $p^2 \mid A$, and $\gcd(A_1, A_2)=1$. By Equation \eqref{zero} if $\gcd(n, A_2) \ne 1$, $Z_{n, A, B}=0$. Hence, 
$$
\zeta_{A, B}(s_1, s_2) = \frac{\zeta(s_1)}{\zeta(2s_1)}\sum_{\substack{n \text{ odd, square-free}\\ \gcd(n, A_2)=1}} a_{n, A, B}(s_1) \frac{L_{2AB}\left(\left(\frac{-nB}{\cdot} \right), s_1\right)}{n^{s_2}}, 
$$
Since $A_1$ is square-free, we can write $A_1 = q_1 \cdots q_r$ with $q_i$'s distinct primes. For each nonempty subset $I$ of $\{q_1, \dots, q_r\}$, let  
$$
\zeta_{A, B}^I(s_1, s_2) =  \frac{\zeta(s_1)}{\zeta(2s_1)}\sum_{\substack{n \text{ odd, square-free}\\ \gcd(n, A_2)=1 \\ \text{for all }q \in I, q \mid n} } a_{n, A, B}(s_1) \frac{L_{2AB}\left(\left(\frac{-nB}{\cdot} \right), s_1\right)}{n^{s_2}}. 
$$
Then by inclusion-exclusion, 
$$
\zeta_{A, B}(s_1, s_2) = \tilde\zeta_{A, B}(s_1, s_2)  + \sum_{I \ne \varnothing} (-1)^{\# I + 1}\zeta_{A, B}^I(s_1, s_2).
$$
So, it suffices to prove the analytic continuation of each $\zeta_{A, B}^I(s_1, s_2)$. Let $t(I) = \prod_{q \in I} q$. Then since $n$ is square-free, writing $n = n' t(I)$, 
\begin{align*}
\zeta_{A, B}^I(s_1, s_2) & =  \frac{\zeta(s_1)}{\zeta(2s_1)}\sum_{\substack{n \text{ odd, square-free}\\ \gcd(n, A_2)=1 \\ \gcd(n', t(I))=1} } a_{n' t(I), A, B}(s_1) \frac{L_{2AB}\left(\left(\frac{-n' t(I) B}{\cdot} \right), s_1\right)}{n^{s_2}} \\ 
& = \frac{\zeta(s_1)}{\zeta(2s_1)}\sum_{\substack{n \text{ odd, square-free}\\ \gcd(n, A_2)=1 \\ \gcd(n', t(I))=1} } a_{n' t(I), A, B}(s_1) \frac{L_{2AB}\left(\left(\frac{-n' t(I) B}{\cdot} \right), s_1\right)}{n^{s_2}}
\end{align*}
The expression $a_{n' t(I), A, B}(s_1)$ is determined with $n'$ modulo $8A/ t(I)$ the square-free part of which has fewer prime factors than the square-free part of $8 A$. By repeating this process we may assume $A=1$, and that, after a change of notation, we have a summation of the form 
\begin{equation}\label{partial}
\frac{\zeta(s_1)}{\zeta(2s_1)} \sum_{n \text { odd, square-free}} A_n(s_1) \frac{L_{2AB}\left(\left(\frac{-n' t(I) B}{\cdot} \right), s_1\right)}{n^{s_2}}
\end{equation}
with meromorphic functions $A_n(s)$ such that if $n_1 \equiv n_2 \mod 8$ then $A_{n_1}(s) = A_{n_2}(s)$. Now an argument similiar to the proof of the analytic continuation of $\tilde\zeta_{A, B}$, with $8$ replacing $8A$, gives the analytic continuation of the zeta function in Equation \eqref{partial}. This gives us the analytic continuation of $\zeta_{A, B}^I(s_1, s_2)$ to the entire $\mathbb C^2$ which implies the analytic continuation of $\zeta_{A, B}(s_1, s_2)$.  Putting everything together we obtain the following theorem: 

\begin{theorem}\label{thm:main}
    The zeta function $Z_{A, B}(s_1, s_2)$ in Equation \eqref{zetaAB} has an analytic continuation to the entire $\mathbb C^2$ as a meromorphic function.  
\end{theorem}

 \bibliographystyle{plain}
 \bibliography{paper}

\end{document}